\documentclass[11pt]{article}
\usepackage{geometry}                		
\usepackage[parfill]{parskip}    		

\usepackage{amssymb}
\usepackage{amsmath}

\usepackage{palatino}

\usepackage{graphicx}
\usepackage[english]{babel}
\usepackage{amscd}
\usepackage{amsthm}

\newtheorem{theorem}{Theorem}[section]
\newtheorem{lemma}[theorem]{Lemma}

\newtheorem{definition}{Definition}
\newtheorem{conjecture}[theorem]{Conjecture}

\newtheorem{problem}[theorem]{Problem}

\title{A generalization of Erd\H{o}s' matching conjecture}

\author{Christos Pelekis \footnotemark[1]
\and
Israel Rocha\thanks{The Czech Academy of Sciences, Institute of Computer Science, Pod Vod\'{a}renskou v\v{e}\v{z}\'{\i} 2, 182 07 Prague, Czech Republic. 
The authors are supported by the Czech Science Foundation, grant number GJ16-07822Y, with institutional support RVO:67985807. E-mails: pelekis.chr@gmail.com;  israelrocha@gmail.com} }

\begin{document}
\maketitle
	
\begin{abstract} 
Let $\mathcal{H}=(V,\mathcal{E})$ be an $r$-uniform hypergraph on $n$ vertices and fix a positive integer $k$ such that $1\le k\le r$. A $k$-\emph{matching} of $\mathcal{H}$ is a collection of edges $\mathcal{M}\subset \mathcal{E}$ such that every subset of $V$ whose cardinality equals $k$ is contained in at most one element of $\mathcal{M}$. The $k$-matching number of $\mathcal{H}$ is the maximum cardinality of a $k$-matching. 
A well-known problem, posed by Erd\H{o}s, asks for the maximum number of edges in an $r$-uniform hypergraph under constraints on its $1$-matching number. In this article we investigate the more general problem of determining the maximum number of edges in an $r$-uniform hypergraph on $n$ vertices subject to the constraint that its $k$-matching number is strictly less than $a$. The problem can also be seen as a generalization of the, well-known, $k$-intersection problem. 
We propose candidate hypergraphs for the solution of this problem, and show that the extremal hypergraph is among this candidate set when $n\ge 4r\binom{r}{k}^2\cdot a$. 
\end{abstract}

\noindent {\emph{Keywords}: Erd\H{o}s' matching conjecture; hypergraphs; complete intersection theorem

\section{Prologue, related work and main results}

A hypergraph, $\mathcal{H}$, is a pair $(V,\mathcal{E})$ where $V$ is a finite set, 
called the \emph{vertex set}, and $\mathcal{E}$ is a collection of subsets of $V$. 
The set $\mathcal{E}$ is called the \emph{edge set} and its elements \emph{edges}. 
We denote by $\binom{V}{k}$ the family consisting of all subsets of $V$ whose cardinality equals $k$. 
A hypergraph is called $r$-uniform if all of its edges have cardinality $r$. 
A hypergraph is called $k$-\emph{intersecting} if the intersection of any two of its edges has cardinality at least $k$. Given a finite set, $F$, we denote by $|F|$ its cardinality and, given a positive integer $m$, we denote by $[m]$ the set $\{1,\ldots,m\}$. A finite set whose cardinality equals $m$ is refer to as an $m$-\emph{set}, for short. 
As an abuse of notation, we sometimes denote by $|\mathcal{H}|$ the number of edges in a hypergraph $\mathcal{H}$. 
A \emph{matching} in a hypergraph, $\mathcal{H}$, is a family of pairwise disjoint edges. The \emph{matching number}
of $\mathcal{H}$, denoted $\nu(\mathcal{H})$,  is the maximum cardinality of a matching. 
The notion of matching is fundamental in combinatorics. Its 
significance is supported by the fact that several combinatorial problems can be reduced to the problem of determining the matching number of appropriate hypergraphs. 

An important problem regarding matchings in uniform hypergraphs was posed by Erd\H{o}s in $1965$, who asked for the determination of the maximum number of edges in an $r$-uniform hypergraph under constraints on its matching number. More precisely, let $\mathcal{H}$ be an $r$-uniform hypergraph on $n$ vertices which satisfies $\nu(\mathcal{H})<a\leq \frac{n}{r}$. 
What is a sharp upper bound on the number of edges in $\mathcal{H}$? 

Erd\H{o}s conjectured that the maximum is attained by two extremal hypergraphs. The first is the hypergraph 
$\mathcal{H}_1$ consisting of all $r$-sets on $ra-1$ vertices, whose matching number is clearly $a-1$. The second one is the $r$-uniform hypergraph, $\mathcal{H}_2$, on $n$ vertices whose edge set consists of all $r$-sets that contain at least one element from a fixed set of $a-1$ vertices, and whose matching number is $a-1$ as well. \\

\begin{conjecture}[Erd\H{o}s' Matching Conjecture, 1965]
\label{conj:1} The number of edges in an $r$-uniform hypergraph, $\mathcal{H}$, on $n$ vertices whose matching number satisfies $\nu(\mathcal{H}) < a\leq \frac{n}{r}$ is at most 
\[ \max\left\{ |\mathcal{H}_1| , |\mathcal{H}_2| \right\}   . \]
\end{conjecture}

When $n\ge (r+1)\cdot a$ is it not difficult to see that $|\mathcal{H}_2| \ge |\mathcal{H}_1|$ 
and therefore, in this case, $\mathcal{H}_2$ is the hypergraph which is conjectured to have the maximum number of edges among all hypergraphs satisfying the assumptions of Conjecture \ref{conj:1}. 
Erd\H{o}s obtained the following result. \\

\begin{theorem}[Erd\H{o}s \cite{Erdos}]
\label{thm:1} There exists some constant $c_{r}$, which depends only on $r$, such that among all $r$-uniform hypergraphs on $n> c_r\cdot a$ vertices that satisfy $\nu(\mathcal{H}) < a$, the hypergraph $\mathcal{H}_2$ has the maximum number of edges. 
\end{theorem}

The problem of determining the smallest value of $c_{r}$ has attracted considerable attention (see \cite{Alonetal, Frankl, Frankl16, Frankl166, Frankletal, FranklRodlRucinski}, among several others). The current best known upper bound  on this constant is $c_r\le 2r+1$, and is due to Frankl \cite{Frankl}. 
Let us also remark that Erd\H{o}s' matching conjecture, if true, has implications in game theory (see \cite{Fokkink}), distributed storage allocation (see \cite[Section 5]{Alonetal}) as well as in probability theory (see \cite{Alonetal, matas}).

In this article we shall be interested in a generalization of Erd\H{o}s' conjecture.  Our work is motivated by   
the following notion of matchings in hypergraphs. \\

\begin{definition}[$k$-matching]
\label{defn:1}
Let $\mathcal{H}=(V,\mathcal{E})$ be an $r$-uniform hypergraph on $n$ vertices and fix a positive integer 
$k$ such that $1\le k\le r$. A $k$-\emph{matching} of $\mathcal{H}$ is a collection of edges $E_{1},\ldots,E_{j}\in\mathcal{E}$ such that every $T\in\binom{V}{k}$ is contained in at most one $E_{i}, i\in\{1,\ldots,j\}$. The maximum cardinality of a $k$-matching in a hypergraph, $\mathcal{H}$, is its $k$-\emph{matching number} and is denoted by $\nu_{k}(\mathcal{H})$. 
\end{definition}

Equivalently, a $k$-matching of $\mathcal{H}=(V,\mathcal{E})$ is a subset $\mathcal{M}\subset \mathcal{E}$ such that $|E_i\cap E_j|\le k-1$, for all $E_i\neq E_j$ from $\mathcal{M}$. 
Let us mention that the notion of $k$-matching arose in the study of certain search games on hypergraphs (see \cite[Appendix C]{thesis}) as well as in the study of certain generalisations of Tuza's conjecture (see \cite{Aharoni}). 

Notice that a $1$-matching of a hypergraph coincides with a matching.  Notice also that when $\nu_k(\mathcal{H})=1$, then any two edges, say $E_1,E_2$, in $\mathcal{H}$ satisfy 
$|E_1\cap E_2|\ge k$ and therefore the problem of maximizing the number of edges in a $r$-uniform hypergraph whose $k$-maching number equals $1$ is equivalent to the problem of maximizing the number of edges in an $r$-uniform $k$-intersecting hypergraph, which we refer in this paper as the $k$-\emph{intersection problem}. 
This problem, having been open for several decades, was proven to be of great importance in the development of extremal set theory (see \cite{FranklToku, Katona}) and was finally resolved by  Ahlswede and Khachatrian  (see \cite{AK}). In particular, the following holds true. \\

\begin{theorem}[see \cite{AK, Wilson}]
\label{EKR}
Let $1\leq k< r$. 
For every positive integer $i$ such that $0\le i \le\frac{n-k}{2}$, let $\mathcal{B}_i$ be the family consisting of all $E\in\binom{[n]}{r}$ such that 
$|E\cap [k+2i]| \ge k+i$. The number of edges in an $r$-uniform $k$-intersecting hypergraph on $n$ vertices is at most 
\[ \max_{0\le i \le\frac{n-k}{2}}\; |\mathcal{B}_i| . \]
Moreover, when $n\ge (k+1)(r-k-1)$  
the number of edges in an $r$-uniform $k$-intersecting hypergraph on $n$ vertices is at most $\binom{n-k}{r-k}$.  
\end{theorem}

The families $\mathcal{B}_i$, defined in Theorem \ref{EKR}, are referred to as \emph{Frankl families} (see \cite{FranklToku}). 
In this work we address the problem of determining the  number of edges in an $r$-uniform hypergraph on $n$ vertices under constraints on its $k$-matching number. More precisely, we examine the following. \\

\begin{problem}
\label{prbl:1}
Let $\mathcal{H}=(V,\mathcal{E})$ be an $r$-uniform hypegraph on $n$ vertices. Fix a positive integer $k$ such that 
$1\le k\le r$ and assume further that  $\nu_k(\mathcal{H}) < a$. What is a sharp upper bound on the number of edges in $\mathcal{H}$?
\end{problem} 

Later, after constructing suitable candidates for the extremal graph of Problem \ref{prbl:1}, we will formulate a generalization of Erd\H{o}s' conjecture which is the main target of this note. Then we verify the conjecture for large values of $n$.

Note that  Problem \ref{prbl:1} is trivial when $k=r$: the maximum number of edges in $\mathcal{H}$ is equal to $\nu_r(\mathcal{H})$. Hence, from now on, we assume that $k<r$. 
Notice also that in Erd\H{o}s' matching conjecture, the extremal hypergraph $\mathcal{H}_2$  is obtained by fixing $a-1$ vertices, say $v_1,\ldots,v_{a-1}$, and then taking all $r$-sets of the vertex set, $V$, that contain at least one of the vertices $v_i,i=1,\ldots,a-1$. 
In the same way, candidates for the extremal hypergraph in Problem \ref{prbl:1} can be obtained as follows. 
 
Fix a set $V$ consisting of $n$ vertices. For every positive integer $i$ such that $0\le i \le \frac{\lfloor \frac{n}{a-1}\rfloor -k}{2}$ and every  family 
$\mathcal{T}=\{T_1,\ldots, T_{a-1}\}\subset \binom{V}{k+2i}$ let $\mathcal{H}(\mathcal{T})$ denote the 
$r$-uniform hypergraph on the vertex set $V$ whose edge set consists of all $E\in \binom{V}{r}$ such that $|E\cap T_j|\ge k+i$, for some $j\in [a-1]$. Notice that the $(k+2i)$-sets $T_j$ need not be disjoint. 
In fact, there are several ways to choose the sets $T_1,\ldots,T_{a-1}$. 
When the sets $T_1,\ldots,T_{a-1}\in \binom{V}{k+2i}$ are \emph{pairwise disjoint}, 
we will refer to the hypergraph $\mathcal{H}(\mathcal{T})$ as a $(n,r,k,a,i)$-\emph{Frankl family}. 
Notice that $(n,r,1,a,i)$-Frankl families are precisely the Frankl families from Theorem \ref{EKR}.    
Observe also that a $(n,r,1,a,0)$-Frankl family is precisely  the hypergraph $\mathcal{H}_2$ in Erd\H{o}s' matching conjecture. 
The following result implies that 
the number of edges in $\mathcal{H}(\mathcal{T})$ is maximized when the sets are disjoint. \\

\begin{theorem}
\label{thm:2}
Let $V=[n]$ be a set of vertices and fix positive integers $r,k,a$ such that $n\ge ra$ and 
$r> k\ge1$. For every positive integer $i$ such that $0\le i \le \frac{\lfloor \frac{n}{a-1}\rfloor -k}{2}$ and every  $\mathcal{T}=\{T_1,\ldots,T_{a-1}\} \subset \binom{V}{k+2i}$ let 
$\mathcal{H}(\mathcal{T})$ be the $r$-uniform hypergraph on the vertex set $V$ whose edge set 
consists of all $E\in \binom{V}{r}$ such that $|E\cap T_j|\ge k+i$, for some $j\in [a-1]$. Then the 
number of edges in $\mathcal{H}(\mathcal{T})$ is less than or equal to the number of edges in a $(n,r,k,a,i)$-Frankl family.  
\end{theorem}

In other words, Theorem \ref{thm:2} suggests that  
candidates for the hypergraph that maximizes the number of edges in 
Problem \ref{prbl:1} can be found among $(n,r,k,a,i)$-Frankl families.
 
We now proceed to find a hypergraph whose $k$-matching number equals $a-1$ and corresponds to the hypergraph $\mathcal{H}_1$ in Erd\H{o}s' conjecture.   
Notice that the hypergraph $\mathcal{H}_1$ is a complete $r$-uniform hypergraph on $ra-1 = r(a-1) + (r-1)$ vertices that has the following property: one can find $a-1$ edges $E_1,\ldots,E_{a-1}\in \mathcal{H}_1$ such that for every $T\in \binom{[ra-1]}{r}$  there exists $i\in \{1,\ldots,a-1\}$ 
such that $|T\cap E_i|\ge 1$.
Similarly, we are looking for an $r$-uniform hypergraph having the property that 
there  exist $a-1$ edges $E_1,\ldots,E_{a-1}$ such that for every $T\in \binom{[ra-1]}{r}$  there exists $i\in [a-1]$ 
such that $|T\cap E_i|\ge k$. Suppose that $r\ge (a-1)(k-1)+1$ and 
let $n_0 = r(a-1)+r -(a-1)(k-1)-1$. Consider the hypergraph, $\mathcal{H}_0$,  
on the vertex set $[n_0]$ whose edge set is $\binom{[n_0]}{r}$. It is not difficult to see that $\mathcal{H}_0$ has the required property. Notice that, when $k=1$, the hypergraph $\mathcal{H}_0$ is the same as the hypergraph $\mathcal{H}_1$. The discussion thus far leads us in the formulation of the following.   \\

\begin{conjecture}  
\label{conj:2}
Let $\mathcal{H}=(V,\mathcal{E})$ be an $r$-uniform hypergraph on $n$ vertices.
Fix a positive integer $k<r$ and assume further that $\nu_k(\mathcal{H})<a$, for some $a\ge 2$, as well as that $n\ge r\cdot a$. 
Set $n_0 = ra -(a-1)(k-1)-1$. 
Then the number of edges in $\mathcal{H}$ is at most 
\[ \max\left\{ \binom{n_0}{r} , |\mathcal{F}_i|; 0\le i \le \frac{\lfloor \frac{n}{a-1}\rfloor -k}{2}  \right\} , \]
where $\mathcal{F}_i$ is a $(n,r,k,a,i)$-Frankl family. 
\end{conjecture}

Notice that when $k=1$ 
the previous conjecture reduces to Conjecture \ref{conj:1}. Notice also that when $\nu_k(\mathcal{H})=1$   
Conjecture \ref{conj:2} reduces to the $k$-intersection problem. 
Hence Conjecture \ref{conj:2} can be seen is a generalization of both Erd\H{o}s' matching conjecture and the $k$-intersection problem. 
In this note we verify the validity of this conjecture for large values of $n$. The proof is by induction on $\nu_k(\mathcal{H})$, where the $k$-intersection problem is the base case. In particular, we obtain the following result. \\

\begin{theorem}
\label{thm:3}
Let $\mathcal{H}=(V,\mathcal{E})$ be a $r$-uniform hypergraph on $n$ vertices. Assume further that 
$\nu_k(\mathcal{H})<a$, where $1\le k < r$, and that $n\ge 4r\binom{r}{k}^2\cdot a$. Let $\mathcal{F}_0$ be a $(n,r,k,a,0)$-Frankl family. 
Then
\[ |\mathcal{E}| \leq |\mathcal{F}_0| . \]
\end{theorem} 
 
The remaining part of this note is organized as follows.  
In Section~\ref{sect:2} we prove Theorem \ref{thm:2}. The proof is probabilistic and is based on a coupling argument.
In Section~\ref{sect:3} we prove Theorem \ref{thm:3} by adapting Erd\H{o}s' proof of Theorem \ref{thm:1} to our setting. 
Section~\ref{sect:4} includes some concluding remarks.

\section{Proof of Theorem \ref{thm:2}}\label{sect:2}

In this section we prove Theorem \ref{thm:2}. The proof is divided into several lemmata and requires some extra piece of notation and definitions. 

Clearly, we may assume that $k\ge 2$ and $a\ge 3$; otherwise there is nothing to prove. 
Let $\mathcal{H}=(V,\mathcal{E})$ be an $r$-uniform hypergraph. 
Let $\mathbb{I}$ be a subset of $V$ whose cardinality equals $r$ which is chosen uniformly at random from the family $\binom{V}{r}$.  Given a set $\mathcal{T}$ consisting of $a-1$ elements $T_1,\ldots,T_{a-1}\in \binom{V}{k+2i}$, 
we say that $\mathbb{I}$ \emph{captures} $\mathcal{T}$  if there exists $j\in [a-1]$ such that $|\mathbb{I}\cap T_j| \ge k+i$.  
 
Now let $\mathcal{T}=\{T_1,\ldots,T_{a-1}\}\subset \binom{V}{k+2i}$ and suppose that the sets $T_j,j=1,\ldots,a-1$ are not disjoint. This means that we can find two $(k+2i)$-sets, say $T_1$ and $T_2$, such that 
$T_1\cap T_2\neq \emptyset$. Let $S = T_1 \cap T_2$ and set $s = |S|$. 
Now choose $s$ vertices $v_1,\ldots,v_s \in V\setminus \cup_{i=1}^{a-1} T_i$ (recall that $n\ge ra$) 
and set $R=\{v_1,\ldots,v_s\}$. 
Now define the family $\mathcal{T}^{\ast} = \{T_1^{\ast},T_2,\ldots,T_{a-1}\}$, where 
$T_1^{\ast} = (T_1\setminus S) \cup R$ and note that $T_1^{\ast}\in \binom{V}{k}$. 
Finally, fix a bijection $\phi:S\to R$, from $S$ onto $R$. 
 
We claim that the number of edges in $\mathcal{H}(\mathcal{T})$ is less than or equal to the number of edges in $\mathcal{H}(\mathcal{T}^{\ast})$. To prove this claim, it is enough to show that 
the probability that $\mathbb{I}$ captures $\mathcal{T}$ is less than or equal to the probability that 
$\mathbb{I}$ captures $\mathcal{T}^{\ast}$. 

Now let $A_1$ be the event  "$\mathbb{I}$ captures $\mathcal{T}$ and does not capture $\mathcal{T}^{\ast}$" and let $A_2$ be the event  
"$\mathbb{I}$ captures $\mathcal{T}^{\ast}$ and does not capture $\mathcal{T}$". \\

\begin{lemma}\label{lem:1} We have $\mathbb{P}[A_1] \le \mathbb{P}[A_2]$. 
\end{lemma}
\begin{proof}
Notice that the event $A_1$ happens if and only if  $|\mathbb{I}\cap T_1|\ge k+i$, $|\mathbb{I}\cap T_{1}^{\ast}|<k+i$ and $|\mathbb{I}\cap T_j|< k+i$, for all $j\in \{2,\ldots,a-1\}$. 
Similarly, the event $A_2$ happens if and only if $|\mathbb{I}\cap T_{1}^{\ast}|\ge k+i$ and $|\mathbb{I}\cap T_j|< k+i$, for all $j\in \{1,2,\ldots,a-1\}$.
Now let $\mathbb{I}$ be an outcome for which the event $A_1$ occurs. 
Set $\mathbb{I}_R = \mathbb{I}\cap R, \mathbb{I}_S = \mathbb{I}\cap S$ and define the set 
\[ J_\mathbb{I} = \left( \mathbb{I} \setminus (\mathbb{I}_S \cup \mathbb{I}_R) \right) \cup (\phi(\mathbb{I}_S)\cup \phi^{-1}(\mathbb{I}_R)). \]
Notice that $\mathbb{J}$ is an outcome for which the event $A_2$ occurs and that $\mathbb{I}_1\neq \mathbb{I}_2$ implies $J_{\mathbb{I}_1}\neq J_{\mathbb{I}_2}$. 
This shows that for every outcome for which $A_1$ occurs we can associate, in an injective way,  an outcome for which $A_2$ occurs. Since the all $r$-sets have the same probability of occurring, the result follows. 
\end{proof}

The previous lemma yields the following.  \\

\begin{lemma}
\label{lem:2} 
We have 
$\mathbb{P}[\mathbb{I}\; \text{captures}\; \mathcal{T}] \leq \mathbb{P}[\mathbb{I}\; \text{captures}\; \mathcal{T}^{\ast}]$. 
\end{lemma}
\begin{proof}
Notice that Lemma \ref{lem:1} yields
\begin{eqnarray*}
 \mathbb{P}[\mathbb{I}\; \text{captures}\; \mathcal{T}] &=&  \mathbb{P}[\mathbb{I}\; \text{captures}\; \mathcal{T}\; \text{and} \; \mathcal{T}^{\ast}]  
 + \mathbb{P}[A_1]  \\
 &\leq& \mathbb{P}[\mathbb{I}\; \text{captures}\; \mathcal{T}\; \text{and} \; \mathcal{T}^{\ast}] 
 + \mathbb{P}[A_2]\\
 &=&  \mathbb{P}[\mathbb{I}\; \text{captures}\; \mathcal{T}^{\ast}] .
\end{eqnarray*}
\end{proof}

Hence Theorem \ref{thm:2} follows upon iterating the previous two lemmata until we get a family $\mathcal{T}$ whose elements are pairwise disjoint. 

In the next section we prove that $(n,r,k,a,0)$-Frankl families are the extremal hypergraphs of Problem \ref{prbl:1} for large values of $n$. 
For the sake of completeness, let us also count the maximum number of edges in such a family.  \\

\begin{lemma}
\label{lem:3}
Let the set $V$ and parameters $n,k,a$ be as in Theorem \ref{thm:2} and let $\mathcal{T}=\{T_1,\ldots,T_{a-1}\}$ be a family consisting of $a-1$ pairwise disjoint $k$-sets from $V$. Then the number of $E\in \binom{V}{r}$ that contain at least one element from $\mathcal{T}$ is equal to 
\[  g(n,r,k,a):= \sum_{j=1}^{\min\{a-1,\left\lfloor r/k\right\rfloor\} } (-1)^{j-1} \binom{a-1}{j}\binom{n-jk}{r-jk}.   \]
\end{lemma}
\begin{proof}  
For $j=1,\ldots,a-1$ let $\mathcal{D}_j$ be the family consisting of all $E\in \binom{V}{r}$ such that $T_j\subset E$.  
The inclusion-exclusion principle yields that $g(n,r,k,a)$ is equal to 
\[ \left|\bigcup_{i=1}^{a-1} \mathcal{D}_i\right| = \sum_{\emptyset \neq J\subset [a-1]} (-1)^{|J|-1} \left| \bigcap_{j\in J}\mathcal{D}_j \right| .  \]
Now notice that for $\emptyset \neq J\subset [a-1]$ we have
\[ \left| \bigcap_{j\in J}\mathcal{D}_j \right| = \binom{n-|J|\cdot k}{r-|J|\cdot k} \] 
and the result follows upon observing that $\left|\bigcap_{j\in J}\mathcal{D}_j \right| =0$, when $|J| > \min\{a-1,\left\lfloor r/k\right\rfloor\}$. 
\end{proof}

\section{Proof of Theorem \ref{thm:3}}\label{sect:3} 

We imitate Erd\H{o}s' proof of Theorem \ref{thm:1}. 
Notice that it is enough to show that if $\mathcal{H}=(V,\mathcal{E})$ is an $r$-uniform hypergraph on $n > 4r\binom{r}{k}^2\cdot a$ vertices such that 
$|\mathcal{E}| \ge 1+g(n,r,k,a)$, where $g(n,r,k,a)$ is as in Lemma \ref{lem:3}, then we have $\nu_k(\mathcal{H})\ge a$. 
We prove this statement by induction on $a$. When $a=2$, then the result follows from the second statement of  Theorem \ref{EKR}. 
Assuming it holds true for $a-1>1$, we prove it  for $a$. 

Let $\mathcal{H}=(V,\mathcal{E})$ be a hypergraph on $n$ vertices which satisfies $|\mathcal{E}| \ge 1+g(n,r,k,a)$. 
For every $T\in \binom{V}{k}$ denote by $d(T)$  the number of edges $E\in \mathcal{E}$ such that $T\subset E$ 
and choose a $k$-set, say $T_1$, for which $d(T_1)$ is maximum. 
We distinguish two cases. 

Assume first that $d(T_1) < \frac{1+g(n,r,k,a)}{(a-1)\binom{r}{k}}$. 
Let $E_1,\ldots,E_l$ be a maximal $k$-matching of $\mathcal{H}$. Notice that this implies that for any $E\in\mathcal{E}$, there exists $j\in [l]$ such that $|E\cap E_j|\ge k$. 
We claim that $l\ge a$. To see this, notice that if $l<a$ then the edges 
$E_1,\ldots,E_l$ would contain at most $(a-1)\binom{r}{k}$ $k$-sets and therefore the total number of edges in $\mathcal{H}$ satisfies $|\mathcal{E}| <1+g(n,r,k,a)$.  Hence $\mathcal{H}$ contains an edge $E_{l+1}$ which satisfies $|E_{l+1}\cap E_j|\le k-1$, for all $j\in [l]$ and contradicts the maximality of $E_1,\ldots,E_l$. Therefore the claim follows and so does Theorem \ref{thm:3}. 

Assume now that $d(T_1) \ge \frac{1+g(n,r,k,a)}{(a-1)\binom{r}{k}}$. Let $\mathcal{H}(T_1)$ be the hypergraph whose vertex set is $V$ and whose edge set, $\mathcal{E}(T_1)$, consists of all $E\in \mathcal{E}$ such that $T_1\nsubseteq E$. Clearly, we have $|\mathcal{E}(T_1)| \ge 1+g(n,r,k,a) - \binom{n-k}{r-k}$. Now notice that 
\[ g(n,r,k,a) - \binom{n-k}{r-k} = g(n,r,k,a-1)\]
and therefore $|\mathcal{E}(T_1)| \ge 1+g(n,r,k,a-1)$. 
The induction hypothesis implies that there exist at least $a-1$ edges $E_1,\ldots,E_{a-1}$ in $\mathcal{H}(T_1)$ such that $|E_i \cap E_j|\le k-1$, for all $i\neq j$.  
Now notice that the proof will follow once we show that there exists $E\in \mathcal{E}$ such that $T_1\subset E$ which does not contain any of the $(a-1)\binom{r}{k}$ $k$-sets that are contained in $E_1,\ldots,E_{a-1}$. 
Let $T$ be a $k$-set which is contained in some $E_i,i=1,\ldots,a-1$. 
Notice that the number of $r$-sets which contain $T_1$ and $T$ 
is at most $\binom{n-|T\cup T_1|}{r-|T\cup T_1|}$. Since $|T\cup T_1|\ge k+1$ it follows that the number of $r$-sets that contain $T_1$ and any of the $k$-sets contained in $E_1,\ldots,E_{a-1}$ is at most 
\[ (a-1)\binom{r}{k}\binom{n-k-1}{r-k-1} . \] 
We now claim that  
\[ d(T_1)  > (a-1)\binom{r}{k}\binom{n-k-1}{r-k-1}, \; \text{when}\; n \ge 4r\binom{r}{k}^2\cdot a,   \] 
which in turn implies that there exists $E\in \mathcal{E}$ such that $T_1\subset E$ and for which 
$|E\cap E_i|\le k-1$, for all $i\in [a-1]$. Hence $\nu_k(\mathcal{H}) \ge a$ and Theorem \ref{thm:3} follows. 
To prove the claim, note that the estimate  
\[ g(n,r,k,a) \ge (a-1)\binom{n-(a-1)k}{r-k} , \]
combined with the assumption $d(T_1) \ge \frac{1+g(n,r,k,a)}{(a-1)\binom{r}{k}}$, 
implies that it is enough to show    
\[  \binom{n-(a-1)k}{r-k}  \ge (a-1)\binom{r}{k}^2\binom{n-k-1}{r-k-1}, \; \text{for}\; n \ge 4r\binom{r}{k}^2\cdot a.   \] 
The last inequality can be equivalently written as  
\[  (n-r) \prod_{i=1}^{(a-2)k-1}\left(1-\frac{r-k}{n-k-i}\right) \ge (a-1)(r-k)\binom{r}{k}^2 .  \]
Now observe that, since $r>k$, we can estimate  
\[  (n-r) \prod_{i=1}^{(a-2)k-1}\left(1-\frac{r-k}{n-k-i}\right) \ge (n-r) \left( 1- \frac{r}{n-ar}\right)^{ar} . \]
Note that the right hand side is an increasing function of $n$, for fixed $r,k$, which upon substituting   
$n = 4r\binom{r}{k}^2\cdot a$ becomes   
\[ \left(4r\binom{r}{k}^2 a -r\right) \left(1- \frac{1}{4\binom{r}{k}^2a - a}\right)^{ar} \ge \left(4r\binom{r}{k}^2 a -r\right) \cdot  \left(1- \frac{1}{4\binom{r}{k}^2a - a}\right)^{4\binom{r}{k}^2 a -a} ,  \]
since $4\binom{r}{k}^2a - a \ge ar$. Since the sequence $\left(1-\frac{1}{m}\right)^m$ is increasing and $4\binom{r}{k}^2 a -a \ge 2$, we conclude
\[ \left(4r\binom{r}{k}^2 a -r\right) \cdot  \left(1- \frac{1}{4\binom{r}{k}^2a - a}\right)^{4\binom{r}{k}^2 a -a} \ge  \left(4r\binom{r}{k}^2 a -r\right)\cdot\frac{1}{4} . \]
Now it is not difficult to verify that 
\[ \left(4r\binom{r}{k}^2 a -r\right) \cdot \frac{1}{4} \ge (a-1)(r-k)\binom{r}{k}^2 \] 
and the claim follows.

\section{Concluding remarks}\label{sect:4}

A well known technique that has been proven to be very fruitful in extremal set theory involves the notion of \emph{shifting} (see \cite{Frankl_shifting, FranklToku}). 

Let $\mathcal{H}=(V,\mathcal{\mathcal{E}})$ be an $r$-uniform hypergraph whose vertex set is indexed by the positive integers, i.e., $V=[n]$ for some positive integer $n$, and fix $1\le i<j\le n$. The $(i,j)$-\emph{shift} of $\mathcal{H}$, denoted $S_{i,j}(\mathcal{H})$, is the hypergraph with vertex set $[n]$ and with edges  
\[
S_{i,j}(E)=\begin{cases}
E\setminus\left\{ i\right\} \cup\left\{ j\right\} , & \text{if }i\in E,j\notin E,E\setminus\left\{ i\right\} \cup\left\{ j\right\} \notin\mathcal{E}\\
E, & \text{otherwise.}
\end{cases}
\]
Clearly, for every $1\le i<j\le n$, the hypergraph $S_{i,j}(\mathcal{H})$ is $r$-uniform and it contains the same number of edges as $\mathcal{H}$. Moreover, it can be shown that the $k$-matching number of $\mathcal{H}$ does not increase under an $(i,j)$-shift. 
It is known (see \cite{Frankl_shifting}) that if we keep on shifting a hypergraph then after a finite number of steps we end up with a \emph{stable} hypergraph, that is, a hypergraph $\mathcal{H}$ for which $S_{i,j}(\mathcal{H})=\mathcal{H}$, for all $1\le i<j\le n$. 

The shifting technique is usually combined with induction on the number of vertices in a stable hypergraph and allows to obtain sharp estimates on the cardinality of the families 
$\mathcal{H}[n] = \{E\setminus \{n\} : n\in E\in \mathcal{E}\}$ and $\mathcal{H}(n)= \{E\in\mathcal{E}: n\notin E\}$ (see \cite{Frankl_shifting}).  However, as can already be seen in the proof of Theorem \ref{thm:3}, Problem \ref{prbl:1} appears to depend on estimates on the cardinality of the families $\mathcal{H}[T] = \{E\setminus T : T\subset E \in \mathcal{E}\}$ and $\mathcal{H}(T) = \{E\in \mathcal{E}: T\nsubseteq E\}$, where $T\in\binom{V}{k}$,  and we were not able to adapt standard shifting arguments in this setting. We believe that suitable generalizations of the notion of shifting may provide improvements upon the constant $4r\binom{r}{k}^2$ in Theorem \ref{thm:3} 
and we hope that we will be able to report on that matter in the future.

\end{document}